\documentclass{amsart}

\usepackage{graphicx}
\usepackage{amsmath}
\usepackage{amssymb}
\usepackage{epstopdf}
\usepackage{xcolor}
\usepackage{subfigure}
\usepackage[pagewise]{lineno}\linenumbers
\usepackage{enumerate}
 \usepackage{amsaddr}
\usepackage{tikz-cd}
\usetikzlibrary{spy}
\usepackage{ctable}
\usepackage{pifont}

\newcommand{\eps}{\varepsilon}
\newcommand{\R}{\mathbb{R}}
\newcommand{\Z}{\mathbb{Z}}
\newcommand{\C}{\mathbb{C}}
\newcommand{\N}{\mathbb{N}}

\newcommand{\Id}{\mathrm{Id}}
\newcommand{\rmi}{\mathrm{i}}
\newcommand{\rme}{\mathrm{e}}
\newcommand{\rmO}{\mathrm{O}}

\newtheorem{Lemma}{Lemma}[section]
\newtheorem*{Lemma*}{Lemma}
\newtheorem{Theorem}{Theorem}
\newtheorem*{Theorem*}{Theorem}
\newtheorem{Proposition}[Lemma]{Proposition}
\newtheorem*{Proposition*}{Proposition}
\newtheorem{Corollary}[Lemma]{Corollary}

\newtheorem{Definition}[Lemma]{Definition}

\newtheorem*{Hypothesis*}{Hypothesis}

 \title[Fredholm properties of singular elliptic operators]{Fredholm properties of singular elliptic operators arising in the study of point defects}
\author{Gabriela Jaramillo }
\address{University of Houston, Department of Mathematics, 3551 Cullen Blvd, Room 641, Houston, TX 77204.}
\email{gabriela@math.uh.edu}
\thanks{This work is supported in part by the National Science
  Foundation under grant DMS- 2307500.}
\nolinenumbers

\begin{document}
\maketitle

\begin{abstract}

Motivated by the dynamics of defects in planar pattern-forming systems, we study Fredholm properties of elliptic operators with singular coefficients in weighted Sobolev spaces.  In particular, we consider a family of doubly weighted spaces that encode algebraic decay/growth of functions at infinity, and near the origin.  Our results give conditions on the weights under which the operators are either injective, surjective, or isomorphisms. We also give a precise description of the kernel and range of these operators.

\end{abstract}

\section{Introduction}
In this paper we investigate properties of the following bounded elliptic operators with singular potentials,

\begin{eqnarray}
 \label{e:Laplace-I}
 \Delta  -1: & H^2_{\sigma,\gamma}(\R^2) \longrightarrow L^2_{\sigma+2, \gamma+2}(\R^2),\\[2ex]
 \label{e:Laplace-1/r-I}
  \Delta - \frac{1}{r^2} -1: & H^2_{\sigma,\gamma}(\R^2) \longrightarrow L^2_{\sigma+2, \gamma+2}(\R^2),\\[2ex]
  \label{e:Laplace0}
 \Delta -1/r^2: & M^{2,2}_{\sigma,\gamma}(\R^2) \longrightarrow L^2_{\sigma+2, \gamma+2}(\R^2).
\end{eqnarray}

Here $H^k_{\sigma, \gamma}(\R^2)$ and $M^{s,p}_{\sigma, \gamma}(\R^2)$ are doubly weighted Sobolev spaces, defined as follows:

\begin{Definition}
Let $\langle x \rangle = (1 +|x|^2)^{1/2}$ and $b(r) \in C^\infty(\R_+)$ with,
\[b(r) = \left \{ \begin{array}{c c c}
1 & \mbox{for} & r>2\\
r & \mbox{for} & r<1.\\
\end{array} \right. \]
Define,
\begin{enumerate}[i)]
\item $M^{s,p}_{\sigma, \gamma}(\R^2, \C)$ as the completion of $C^\infty_0(\R^2,\C)$ functions under the norm
\[ \|u\|_{M^{s,p}_{\sigma, \gamma}} = \sum_{|\alpha|\leq s} \| D^\alpha u \; b(|x|)^{\sigma +|\alpha|} \langle x \rangle^{\gamma +|\alpha|} \|_{L^p}, \]
and
\item  $H^{s}_{\sigma, \gamma}(\R^2, \C)$ as the completion of $C^\infty_0(\R^2,\C)$ functions under the norm
\[ \|u\|_{H^{s}_{\sigma, \gamma}} = \sum_{|\alpha|\leq s} \| D^\alpha u \; b(|x|)^{\sigma +|\alpha|} \langle x \rangle^{\gamma } \|_{L^2}. \]

\end{enumerate}
\end{Definition}

Our results give conditions on the weights $\sigma$ and $\gamma$, under which  the operators \eqref{e:Laplace-I}, \eqref{e:Laplace-1/r-I}, and \eqref{e:Laplace0} are Fredholm mappings,  i.e have closed range, finite dimensional kernel and cokernel. We also provide an explicit description of the kernel and cokernel for these operators.
These results are summarized in  Theorems \ref{t:helmholtz} and \ref{t:Laplace} stated at the end of this introduction and in Corollary \ref{c:Fredholm_helmholtz} stated in Section \ref{s:helmholtz}.

Our motivation for considering the above elliptic operators  comes from pattern forming systems where point defects play a crucial role in shaping otherwise uniform patterns.
  For instance, these elliptic maps appear when using perturbation methods to study point defects in planar systems, including dislocations and vortices  in the real Ginzburg-Landau equation  \cite{chiron2021, pismen1990, weinan1994}, and spiral waves its  complex counterpart \cite{greenberg1980, greenberg1981, kopell1981, aguareles2023rigorous}.  
The results presented here also have implications for more general model equations like the Swift-Hohenberg equation, which describes the formation of striped and spot patterns in systems undergoing a Turing bifurcation. Indeed, Fredholm properties for related elliptic operators were used in \cite{Jaramillo2019, scheel2024pinning} to determine how spatially localized impurities can alter striped patterns.

As a first  concrete example of where these operators appear, consider the real Ginzburg-Landau equation,
\[ \partial_t A = \Delta A + A - |A|^2 A, \quad x \in \R^2, \quad A \in \C.\]
 In the context of liquid crystals,  this amplitude equation describes the emergence of roll patterns near the onset of convection.
Vortex solutions, $A_*(r, \theta) = \rho(r) \rme^{\rmi \theta}$, where $\rho(r)$ satisfies the equation
\begin{equation}\label{e:herve}
 \partial_{rr} \rho + \frac{1}{r} \partial_r \rho - \frac{1}{r^2} \rho + \rho - \rho^3=0, \quad \rho(0) = 0, \quad \rho(\infty) =1,
 \end{equation}
 then help us describe  dislocations. Indeed, while vortices themselves are stationary solutions,  they can travel in a direction parallel to rolls when embedded in a uniform striped pattern. As they move, these defects modify the pattern allowing solutions with different far-field wave-numbers to co-exist, creating a dislocation.

To analyze these point defects, one can view them as particles interacting with a `force field' described by the roll pattern. The rules that govern these interactions can then be understood through matching short scale solutions near the defect's core with far-field approximations of the modified roll pattern.
To derive these approximations one can look first for stationary solutions which are perturbations of a vortex,
leading one to study the linearization, $L$,  of the real Ginzburg-Landau equation about $A_*$. 
Or, equivalently,
\begin{equation}\label{e:linearGL}
 \rme^{-\rmi \theta} L \rme^{\rmi \theta} = \mathcal{L} = 
\begin{bmatrix}
\Delta - \frac{1}{r^2} & -\frac{2}{r^2} \partial_\theta\\[2ex]
\frac{2}{r^2} \partial_\theta & \Delta - \frac{1}{r^2} 
\end{bmatrix} 
+
\begin{bmatrix}
1- 3 \rho^2 & 0 \\[2ex]
0 & 1- \rho^2
\end{bmatrix} 
\end{equation}
(see \cite{pacard2012linear} for a derivation of the above linearized operator).

The elliptic operators referenced in expressions \eqref{e:Laplace-I} and \eqref{e:Laplace0}  then emerge when one takes into account the behavior of solutions, $\rho$, of equation \eqref{e:herve} near the core and at infinity. 
More precisely, using the expansions,
\begin{align*}
 \rho(r) & \sim a r - \frac{a}{8} r^3 + \rmO(r^5) \quad \mbox{as} \quad r\to 0\\
 1- \rho(r)^2 & \sim \frac{1}{r^2}  + \rmO(1/r^4) \quad \mbox{as } \quad r \to \infty
 \end{align*}
where  $a$ is a positive constant (see \cite{herve1994etude}), one finds that on bounded sets containing the origin the operator in expression \eqref{e:linearGL} corresponds to a compact perturbation of
\[
\mathcal{L}_0 =  \begin{bmatrix}
\Delta - \frac{1}{r^2} & -\frac{2}{r^2} \partial_\theta\\[2ex]
\frac{2}{r^2} \partial_\theta & \Delta - \frac{1}{r^2} 
\end{bmatrix}.
\] 
On the other hand, at infinity terms of order $\rmO(1/r^2)$ are compact, leading one to study the far field operator
\[ \mathcal{L}_f = 
\begin{bmatrix}
\Delta - 2 & 0\\
0 & \Delta
\end{bmatrix}.\]

The above heuristics can be made rigorous by choosing appropriate spatial scalings. Core and far-field approximations of  point defects can then be obtained exploiting Fredholm properties of the above operators together with the implicit function theorem, or Lyapunov-Schmidt reduction.

As a second example, consider now the complex Ginzburg Landau equation,
\[\partial_t A = ( 1+ \rmi \alpha) \Delta A + A - (1 - \rmi \beta) |A|^2 A , \quad A \in \C, \quad x \in \R^2\]
with $\alpha, \beta \in \R$. Spiral waves correspond to solutions 
\[A(r,\theta) = (1+ \rho(r)) \rme^{\rmi ( \Omega + \theta + \phi(r))},\]
with the amplitude, $\rho$, now satisfying the boundary conditions $\rho(0) = \rho(\infty) =0$, while the phase, $\phi$, satisfies $\phi(r) \sim k r$ in the far field, for some $k>0$.
In addition, one has that the frequency $\Omega = \Omega(k)$ follows a dispersion relation, implying  that both $k$ and $\Omega$ are selected by the system. 

To prove existence of these defect solutions, one can insert this ansatz into the complex Ginzburg-Landau equation, and again use the implicit function theorem to obtain functions $\rho$ and $\phi$ satisfying the prescribed boundary conditions.
This approach then leads us to study Fredholm properties of the linear map
\[ \mathcal{L} = 
\begin{bmatrix}
\Delta - \frac{1}{r^2} -2 & - \alpha \Delta \\
\beta & \Delta
\end{bmatrix}\]
where we find the elliptic operator shown in expression \eqref{e:Laplace-1/r-I} in the upper left corner.

%%%%%%%%%%%%%%%%%
In the above examples, a key assumption  is the separation between  core and far-field behavior of solutions, which allows one to treat point defects as particles/perturbations. 
This viewpoint is common in the physics and material science literature where it has been used  to study, in formal way, dislocations and vortices in liquid crystals and superconducting materials \cite{pismen1990, weinan1994}.
More recently, these ideas were used in \cite{ortner2016} to rigorously obtain a far-field description of the elastic field created by lattice defects in crystalline structures, and to then  leverage this information to run simulations \cite{Luskin2013}. 

In the context of pattern formation the above approach is less common. Indeed, most analysis of defects and coherent-structures is done using methods from spatial dynamics (see for example \cite{sandstede2004, kollar2007, sandstede2023spiral}).
Part of the issue is that model equations for pattern forming systems are usually set up on unbounded domains in order to eliminate boundary effects. As a result, when defined over standard Sobolev spaces, the elliptic operators considered here have a zero eigenvalue embedded in their essential spectrum. Thus, obtaining Fredholm properties for these maps is a non-trivial task.
Nevertheless, this method has been successfully applied to  problems arising in the study of fluid flows  on unbounded domains. For instance, in references \cite{galdi2000} and in \cite{Galdi2013},  existence of stationary solutions to the Navier-Stokes problem posed on exterior domains and past a rotating body, respectively, is obtained by deriving Fredholm properties for the relevant linear maps. These results are in turn then used to prove existence of a family of periodic solutions that bifurcate from these steady states, see \cite{galdi2016} and \cite{Galdi2023}.

%%%%%%%%%%%%%

Notice that to bypass the difficulties generated by the zero eigenvalue, selecting the correct domain for the operator is important.
The use of doubly weighted spaces in this paper is therefore not an arbitrary choice. As shown below, these spaces encode the behavior at infinity and near the origin of elements in the kernel of the operator.
These elements are responsible for the generation of Weyl sequences, which prevent the operators from having a closed range. By suitable picking the weights $\sigma, \gamma$, one can eliminate these sequences from the domain of the operator and thus recover Fredholm properties for these maps.
This same strategy was used by McOwen and Lockhart, where similar weighted spaces are used to derive Fredholm properties for the Laplacian and elliptic operators with algebraically decaying coefficients, see \cite{mcowen1979} and \cite{lockhart1985} respectively. In both works, the weighted Sobolev spaces  considered restrict the level of decay or growth of functions only at infinity, but not near the core, as is done in this paper. Similar results were also obtained for the Heat operator by Hile and Mawata in \cite{hile1998}, where again suitable weighted spaces are used in order to  obtain a closed range operator.

Finally, we remark that weighted H{\"o}lder spaces that encode the same type of singularities at the origin as the Sobolev spaces presented here have been used previously in applications.
For instance, these spaces are used in \cite{cafarelli1984, hardt1985} to prove the existence of a minimal surface in $\R^{n+1}, (n\geq 3)$ which contains an isolated singularity and that is not a cone.
In \cite{smale1989, smale1993} variations of these spaces are used to construct minimal hypersurfaces with an arbitrary number of isolated singularities (with $n \geq 7$), and in \cite{mou1989, hardt1992} these spaces are used to study harmonic functions between smooth manifolds and which admit a set of singular points.

We conclude this  introduction with two theorems summarizing the results of the paper. In Section \ref{s:preliminaries} we give some background information on the weighted Sobolev spaces we consider. Section \ref{s:helmholtz} and contain the proof of Theorem \ref{t:helmholtz} and Corollary \ref{c:Fredholm_helmholtz}, while the prove of Theorem \ref{t:Laplace} is given in Section \ref{s:laplace}.  Additional inequalities  are presented in the Appendix.

\begin{Theorem}\label{t:helmholtz}
Let $\sigma, \gamma \in \R$ with $\sigma \notin \Z$. Then, the operator
\[ \Delta - \Id: H^2_{\sigma, \gamma}(\R^2) \longrightarrow L^2_{\sigma +2, \gamma}(\R^2)\]
is Fredholm. In particular, fixing $n \in \N \cup \{0\}$ and picking
\begin{itemize}
\item $n-1 < \sigma < n$, the operator is surjective with kernel spanned by the set,
\[ M_n =\left\{ K_m(r) \cos(m \theta), \quad K_m(r) \sin (m \theta) : \quad m = 0,1,2, \cdots n \right\}.\]
\item while for $-n-2 < \sigma < -n-1$, the operator is injective with cokernel spanned by $M_n$, defined as above.
\end{itemize}
Here $K_m(r)$ denotes the modified Bessel function of the second kind and of order $m$.

\end{Theorem}

\begin{Theorem}\label{t:Laplace}
Let $n \in \Z$ and consider $\sigma, \gamma \in \R$ such that $\sigma +1 \neq \pm \sqrt{n^2 +1}$ and 
$\gamma +1 \neq \pm \sqrt{n^2 +1}$. Then, the operator
$$ \Delta -1/r^2: M^{2,2}_{\sigma,\gamma}(\R^2) \longrightarrow L^2_{\sigma+2, \gamma+2}(\R^2)$$
is Fredholm. In addition, letting  $k, n, m \in \N \cup \{0\}$,  $q(k) = \sqrt{k^2+1}$, and defining,
\[ N_n = \left\{ r^{q(k)} \cos( k \theta),\; r^{q(k)} \sin( k \theta) \Big | \quad k = 0,1,2, \cdots n \right \},\]
\[ M_m =\left\{  r^{-q(k)} \cos( k \theta),\; r^{-q(k)} \sin( k \theta) \Big | \quad k = 0,1,2, \cdots m \right\}.\]
we have that:
\begin{enumerate}[i)]

\item If $\sigma >-2$ and $\gamma<0$, such that
\[ - q(n+1)< \gamma+1< -q(n) \quad \mbox{and} \quad q(m) < \sigma + 1< q(m+1), \]
then the operator is surjective with a finite dimensional kernel spanned by $N_n \cup M_m$.
\item If $\sigma <-2$ and $\gamma>0$, 
such that
\[ q(n) < \gamma+1< q(n+1) \quad \mbox{and} \quad -q(m+1) < \sigma+1< -q(m),\]
the operator is injective with a finite dimensional co-kernel spanned by $N_n \cup M_m$.

\item If $\sigma >-2$ and $\gamma >0$ such that
\[ q(n) < \gamma+1< q(n+1) \quad \mbox{and} \quad q(m) < \sigma +1< q(m+1),\]
then the operator has a finite dimensional kernel spanned by $M_m$ 
and a finite dimensional co-kernel spanned by $N_n$.
\item If $\sigma <-2$ and $\gamma<0$ such that
\[ -q(n+1) < \gamma+1< -q(n) \quad \mbox{and} \quad -q(m+1) < \sigma+1< -q(m),\]
then the operator has a finite dimensional kernel spanned by $N_n$
 and a finite dimensional co-kernel spanned by $M_m$.
\end{enumerate}
Finally, if $\gamma +1 = \pm \sqrt{n^2 +1}$ or $\sigma +1 = \pm \sqrt{n^2 +1}$ for some $n \in \Z$, the operator does not have a closed range.
\end{Theorem}

%%%%%%%%%%%%%%%%%%%%%%%%%%%%%%%%%%%%%%%%%%%%%%
%%%%%%%%%PRELIMINARIES
%%%%%%%%%%%%%%%%%%%%%%%%%%%%%%%%%%%%%%%%%%%%%%

\section{Preliminaries}\label{s:preliminaries}

In this section we recall the definition of the Sobolev spaces we use in this paper and state some of their relevant properties.

We define $M^{s,p}_{\sigma, \gamma}(\R^2, \C)$ and $H^{s}_{\sigma, \gamma}(\R^2, \C)$
 as the completion of $C^\infty_0(\R^2,\C)$ functions under the respective norms
\[ \|u\|_{M^{s,p}_{\sigma, \gamma}} = \sum_{|\alpha|\leq s} \| D^\alpha u \; b(|x|)^{\sigma +|\alpha|} \langle x \rangle^{\gamma +|\alpha|} \|_{L^p}, \]
\[ \|u\|_{H^{s}_{\sigma, \gamma}} = \sum_{|\alpha|\leq s} \| D^\alpha u \; b(|x|)^{\sigma +|\alpha|} \langle x \rangle^{\gamma } \|_{L^2}, \]
where $\langle x \rangle = (1 +|x|^2)^{1/2}$ and $b(r) \in C^\infty(\R_+)$ with,
\[b(r) = \left \{ \begin{array}{c c c}
1 & \mbox{for} & r>2\\
r & \mbox{for} & r<1.\\
\end{array} \right. \]
Notice that equivalent norms for these space are
\[ \|u\|_{M^{s,p}_{\sigma, \gamma}} = \sum_{|\alpha|\leq s} \| D^\alpha u \; b(|x|)^{\sigma +|\alpha|} \|_{L^p(B_1)} + \| D^\alpha u \;\langle x \rangle^{\gamma +|\alpha|} \|_{L^p(\R^2 \setminus B_1)}, \]
and
\[ \|u\|_{H^{s}_{\sigma, \gamma}} = \sum_{|\alpha|\leq s} \| D^\alpha u \; b(|x|)^{\sigma +|\alpha|} \|_{L^2(B_1)} + \| D^\alpha u \;\langle x \rangle^{\gamma } \|_{L^2(\R^2 \setminus B_1)} \]
where $B_1$ denotes the unit ball in $\R^2$.

Extending the theory presented in \cite{stein1971} one can show that equivalent formulations for these doubly weighted spaces are,
\begin{equation}\label{e:directsum1}
 M^{s,p}_{\sigma, \gamma}(\R^2,\R) = \oplus_n m^{s,p}_{n, \sigma, \gamma}, \quad n \in \Z, 
 \end{equation}
 and
 \begin{equation}\label{e:directsum2}
 H^{s}_{\sigma, \gamma}(\R^2,\R) = \oplus_n h^{s,p}_{n, \sigma, \gamma}, \quad n \in \Z. 
 \end{equation}
In the first case,
\[ m^{s,p}_{n, \sigma, \gamma}  = \{ u \in M^{s,p}_{\sigma, \gamma}(\R^2,\C) : u = u_n(r) \rme^{\rmi n \theta}, \; u_n(r) \in M^{s,p}_{r,\sigma, \gamma}(\R^2,\R)\}.\]
while 
\[ h^{s,p}_{n, \sigma, \gamma}  = \{ u \in H^{s}_{\sigma, \gamma}(\R^2,\C) : u = u_n(r) \rme^{\rmi n \theta}, \; u_n(r) \in H^{s}_{r,\sigma, \gamma}(\R^2,\R)\}.\]
In the above definitions we use the notation $M^{s,p}_{r,\sigma, \gamma}(\R^2,\R)$ and $H^{s}_{r,\sigma, \gamma}(\R^2,\R)$  to describe those functions in $M^{s,p}_{\sigma, \gamma}(\R^2,\R)$ and $H^{s}_{\sigma, \gamma}(\R^2,\R)$, respectively, which are radially symmetric.
For a proof of \eqref{e:directsum1} see \cite{jaramillo2022}. The second formulation \eqref{e:directsum2} follows by similar arguments.

%%%%%%%%%%%%%%%%%%%%%%%%%%%%%%%%%%%%%
%%%%%%%%%%%%%LAPLACE - I %%%%%%%%%%%%%%%%%%
%%%%%%%%%%%%%%%%%%%%%%%%%%%%%%%%%%%%%

\section{Fredholm Properties for $\Delta- \Id$}\label{s:helmholtz}
We first prove Freholm properties for the operator 
\[ \Delta -\Id: H^2_{\sigma,\gamma}(\R^2) \longrightarrow L^2_{\sigma+2, \gamma+2}(\R^2).\]
We split the result into two cases depending on the value of $\sigma $ .
We first show that for $\sigma <-1$ and $\gamma \in \R$, the operator is injective with a finite dimensional cokernel. This result is summarized in Proposition \ref{p:injective}.
We then use the fact that the operator is self-adjoint to prove in Proposition \ref{p:surjective} that for $\sigma >-1$ and $ \gamma \in \R$, the operator is surjective with finite dimensional kernel.
As a consequence we obtain Theorem \ref{t:helmholtz} stated in the introduction.

To analyze the operator we use the decomposition of $H^2_{\sigma, \gamma}(\R^2)$ into the direct sum $\oplus h^2_{n,\sigma, \gamma}$ and write
\begin{align*}
(\Delta - \Id) u = & f\\
\sum_n ( \Delta_n - \Id)u_n(r)\; \rme^{\rmi n \theta} =& \sum_n f_n(r) \;\rme^{\rmi n \theta} 
\end{align*}
where $\Delta_n = \partial_{rr} + \frac{1}{r} \partial_r - \frac{n^2}{r^2}$.
This view point allows to directly extend the results of Theorem \ref{t:helmholtz} to analyze the related equation
\begin{align*}
(\Delta - \frac{1}{r^2} - \Id) u = & f\\
\sum_n ( \Delta_n - \frac{1}{r^2} -\Id)u_n(r)\; \rme^{\rmi n \theta} =& \sum_n f_n(r) \;\rme^{\rmi n \theta} 
\end{align*}
and thus obtain the following corollary to Theorem 1.

\begin{Corollary}\label{c:Fredholm_helmholtz}
Let $\sigma, \gamma \in \R$ with $\sigma \notin \{ \sqrt{n^2+1} \mid n \in \N \cup \{0\} \} $. Then, the operator
\[ \Delta - \frac{1}{r^2} - \Id: H^2_{\sigma, \gamma}(\R^2) \longrightarrow L^2_{\sigma +2, \gamma}(\R^2)\]
is Fredholm. In particular, fixing $n \in \N \cup \{0\}$ and picking
\begin{itemize}
\item $\sqrt{n^2+1} -1 < \sigma < \sqrt{n^2+1} $, the operator is surjective with kernel spanned by the set,
\[ M_n =\left\{ K_{q(m)}(r) \cos(m \theta), \quad K_{q(m)}(r) \sin (m \theta) \Big | \; q(m) = \sqrt{m^2+1},  \quad m = 0,1,2, \cdots n \right\}.\]
\item while for $- \sqrt{n^2+1} -2 < \sigma < - \sqrt{n^2+1} -1$, the operator is injective with cokernel spanned by $M_n$, defined as above.
\end{itemize}
Here $K_{q(m)}(r)$ denotes the modified Bessel function of the second kind and of order $q(m)$.

\end{Corollary}

We consider first the case when $\sigma <-1$.
In Lemma \ref{l:surjective_helmholtz} we first show that the operators $(\Delta_n -\Id): H^2_{r,\sigma, \gamma}(\R^2) \longrightarrow L^2_{r,\sigma, \gamma}(\R^2) $ are injective and have a closed range, provided $\sigma <-n-1$.
Next, in Lemma \ref{l:invertible_helmoltz} we show that for $ -n-1 < \sigma < n-1$ these operators are injective.

\begin{Lemma}\label{l:surjective_helmholtz}
Fix $n \in \N \cup \{0\}$ and take $\sigma<-n-1$.
Then, the operator
\[ \Delta_n - \Id: H^2_{r,\sigma, \gamma}(\R^2) \longrightarrow L^2_{r,\sigma +2, \gamma}(\R^2)\]
is injective, with closed range, and finite dimensional cokernel spanned by $K_n(r)$,
 the modified Bessel function of the second kind and of order $n$.
\end{Lemma}

\begin{proof}
Fix $n \in \N \cup \{0\}$, pick $\sigma < -(n+1)$ and let
\[ R = \left \{ f \in  L^2_{r,\sigma+2,\gamma}(\R^2) \Big | \int_0^\infty f(r) K_n(r) \; rdr  =0 \right \}.\]
Notice that the set $R$ is closed. Indeed, this follows since for $\sigma < -( n+1)$ 
the function $K_n(r)$ defines a bounded linear functional on
$L^2_{r,\sigma+2,\gamma}(\R^2)$. The set $R$ is the nullspace of this functional and is therefore closed.

To show that the operator given in the proposition has closed range and finite dimensional cokernel,  we prove that the map
\[ \begin{array}{c c c}
R & \longrightarrow & H^2_{r,\sigma,\gamma}(\R^2)\\
f & \longmapsto & G \ast f
\end{array}
\]
given by
\begin{equation}\label{e:inverseu}
G \ast f = I_n(r) \int_r^\infty K_n(\rho) f(\rho) \rho \;d\rho +
  K_n(r) \int_0^r I_n(\rho) f(\rho) \rho \;d\rho.
  \end{equation}
gives a representation for the inverse $(\Delta_n -\Id)^{-1}$ and is a bounded operator. 

The first statement follows from realizing that the equation
\[ \partial_{rr} u + \frac{1}{r} \partial_r u - \frac{n^2}{r^2} u - u = f \]
can be solved using the Green's function
\[ G(r, \rho) = \left \{ 
\begin{array}{c c c}
\dfrac{I_n(r) K_n(\rho)}{W(\rho)} & \mbox{for} & 0< r< \rho,\\[2ex]
\dfrac{K_n(r) I_n(\rho)}{W(\rho)} & \mbox{for} & \rho< r< \infty,
\end{array}
\right.\]
where the Wronskian $W(\rho) =   I'_n(\rho) K_n(\rho)- I_n(\rho) K'_n(\rho) = 1/\rho$,
giving us expression \eqref{e:inverseu}.

To prove that the operator $G \ast $ is a bounded map with domain $H^2_{r,\sigma,\gamma}(\R^2)$,
we let $u = G \ast f$ and in the following paragraphs show that 
\begin{equation}\label{e:ineq_h1}
\| u\|_{H^1_{r,\sigma,\gamma}(\R^2)} \leq \|f\|_{L^2_{r,\sigma+2, \gamma}(\R^2)}.
\end{equation}
Then, to conlcude that $u$ is in $H^2_{r,\sigma,\gamma}(\R^2)$, we
let $v(r,\theta) = u(r) \rme^{\rmi n \theta}$ and observe that
\[ (\Delta -\Id)u(r) \rme^{\rmi n \theta} = [(\Delta_n -\Id) u(r)] \rme^{\rmi n \theta} = f(r)  \rme^{\rmi n \theta}.\]
The following inequality
\[ \|v\|_{H^2_{\sigma,\gamma}(\R^2)} \leq C(\| (\Delta -\Id)v\|_{L^2_{\sigma+2,\gamma}(\R^2)} + \| v\|_{L^2_{\sigma,\gamma}(\R^2} )\]
stated in Lemma \ref{l:interpolation_2} in the Appendix,  gives us the desired result.

We split the proof of inequality \eqref{e:ineq_h1} into three main steps.

\begin{table}[t]
\begin{center}
\begin{tabular}{ m{2cm} m{5cm} m{4.5cm}  } 
\specialrule{.1em}{.05em}{.05em} 
  & $z \to 0$ & $z \to \infty $\\
  \hline
  \vspace{1ex}
$ K_\nu(z) $ &$ \sim \frac{1}{2}\Gamma(\nu)\left[\frac{z}{2} \right]^{-\nu}$ &$ \sim \sqrt{ \frac{\pi}{2 z} } \rme^{-z} $\\ 
$ I_\nu (z) $ & $ \sim \frac{1}{\Gamma(\nu+1)} \left[ \frac{z}{2} \right]^\nu $ &$\sim \sqrt{ \frac{1}{2 z \pi} } \rme^{z}  $\\      
\specialrule{.1em}{.05em}{.05em}
\end{tabular}
\end{center}
\caption{ 
Asymptotic behavior for the first-order Modified Bessel functions of the first and second kind, and their derivatives, taken from \cite[(9.6.8), (9.6.9), (9.7.2)]{abramowitz}. Here, the integer $\nu\geq1$ and $\Gamma(\nu)$ is the Euler-Gamma function. }
\label{t:bessel2}
\end{table}

{\bf Step 1:} To  bound $u$ in $L^2_{r,\sigma,\gamma}(\R^2)$, we show that
\begin{align}\label{e:ineq1} 
\| u\|_{L^2_{r,\sigma,\gamma}(\R^2 \setminus B_1)}  \leq \| f\|_{L^2_{r,\sigma+2, \gamma}(\R^2) } \\
\label{e:ineq2}
\| u\|_{L^2_{r,\sigma,\gamma}(B_1)} \leq \| f\|_{L^2_{r,\sigma+2, \gamma}(\R^2) }. 
\end{align}

{\bf Step 1.1:} To prove inequality \eqref{e:ineq1} we write $ \| u\|_{L^2_{r,\sigma,\gamma}(\R^2 \setminus B_1)} < A +B$ where
\begin{align*}
A^2 = & \int_1^\infty \left[I_n(r)  \int_r^\infty K_n(\rho) f(\rho) \; \rho d\rho \right]^2 \langle r \rangle^{2 \gamma} \;r dr,\\
B^2 = & \int_1^\infty \left[K_n(r)  \int_0^r I_n(\rho) f(\rho) \; \rho d\rho \right]^2 \langle r \rangle^{2 \gamma} \;r dr,
\end{align*}
and show that $A,B \leq \|f \|_{L^2_{r,\sigma+2, \gamma}(\R^2\setminus B_1)}$.

Because we are in the setting where $1<r< \rho$ we can use the asymptotic forms of the functions $I_n(r)$ and $K_n(r)$ (see Table \ref{t:bessel2}) to write
\[
A^2 \leq C  \int_1^\infty \left[ \int_r^\infty  \frac{\rme^{- (\rho-r)}}{\sqrt{r} \sqrt{\rho}} f(\rho); \rho d\rho \right]^2 \langle r \rangle^{2 \gamma} \;r dr,
\]
while for the term $B^2$ we have that
\[B^2 \leq  \int_1^\infty \left[ 
 \int_0^r \frac{\rme^{-(r-\rho)} }{ \sqrt{r} \sqrt{\rho}} f(\rho) \; \rho d\rho
 \right]^2 \langle r \rangle^{2 \gamma} \;r dr,\]
since $r>1$ and $0< \rho <r$.
An application of Young's inequality then gives us the desired results:
\begin{align*}
A & \leq C(\gamma) \| f\|_{L^2_{r,\sigma+2,\gamma}(\R^2\setminus B_1)},\\
B & \leq C(\gamma) \| f\|_{L^2_{r,\sigma+2,\gamma}(\R^2\setminus B_1)}.
\end{align*}

{\bf Step 1.2:} To prove inequality \eqref{e:ineq2} we follow a similar strategy and write $\| u\|_{L^2_{r,\sigma,\gamma}(B_1)} \leq E + D$ where
\begin{align*}
E^2 = & \int_0^1\left[I_n(r)  \int_r^\infty K_n(\rho) f(\rho) \; \rho d\rho \right]^2 r^{2 \sigma} \;r dr\\
D^2 = & \int_0^1 \left[K_n(r)  \int_0^r I_n(\rho) f(\rho) \; \rho d\rho \right]^2 r^{2 \sigma} \;r dr
\end{align*}
and show that $E,D \leq \|f \|_{L^2_{r,\sigma+2, \gamma}(B_1)}$.\\
 We treat the case of $n\geq 1$ first, and show the result for $n=0$ in Step 1.3.

We start with the term $E$. Because $0<r<1$ we may use the asymptotic approximation for $I_n(r)$ near the origin,
 and because $f \in R$, we may write
\[E^2 \leq  \int_0^1\left[r^n  \int_0^r - K_n(\rho) f(\rho) \; \rho d\rho \right]^2 r^{2 \sigma} \;r dr.\]
Notice that letting $g(\rho) = - K_n(\rho) f(\rho) \rho $ and $v (r) = \int_0^r g(\rho) \; d\rho,$ we have that $v$ solves 
\begin{equation}\label{e:v_eq}
 \partial_r v = g(r) \qquad v(0) =0.
\end{equation}
The above inequality can then be written as
\[ E^2 \leq \| r^n v(r) \|^2_{L^2_{r,\sigma,\gamma}(B_1)} =\|  v(r) \|^2_{L^2_{r,\sigma+n ,\gamma}(B_1)}  \]
and our goal now is to show that
$$\|  v(r) \|^2_{L^2_{r,\sigma+n ,\gamma}(B_1)}  \leq \| f\|^2_{L^2_{r,\sigma+2, \gamma}(\R^2)}.$$

Using the change of variables
\begin{equation}\label{e:change_coor}
\tau = \ln r, \qquad w = v( \rme^{\tau}) \rme^{-\beta \tau}, \qquad h(\tau) = g(\rme^{\tau}) \rme^{-\beta \tau }\rme^{\tau},  
\end{equation}
with $\beta =  -[\sigma+n +1]$ we see that equation \eqref{e:v_eq} is now given by
\begin{equation}\label{e:w_eq}
 \partial_\tau ( w \rme^{\beta \tau}) = h(\tau) \rme^{\beta \tau}, \qquad \lim_{\tau \to -\infty} w(\tau)e^{\beta \tau} =0.
 \end{equation}
 This ode has solutions
 \begin{equation}\label{e:2_sol}
 w(\tau) = \int_{-\infty}^\tau \rme^{\beta(s-\tau)} h(s) \;ds
 \end{equation}
 and one can show, again using Young's inequality, that
  \begin{equation}\label{e:ineq_w}
  \| w \|_{L^2(-\infty,0)} \leq C(\beta) \| h\|_{L^2(-\infty,0)},
  \end{equation}
  provided $\beta  = - [\sigma +n +1]>0$.
 
 We now relate this last result to the original variables $v$, $g$, and $f$.
 First, a short calculation shows that with $\beta=-  [\sigma +n +1]$,  $v$ and $w$ satisfy
\begin{equation}\label{e:norm_w}
 \|w\|_{L^2(-\infty,0)} = \| v\|_{L^2_{r,\sigma+n, \gamma}(B_1)}
\end{equation}
Next, we look the $L^2$ norm of $h$ and find
\[
 \|h \|_{L^2(-\infty,0)} = \int_0^1 |g|^2 r^{-2 \beta} r^2 \frac{1}{r} \;dr = \int_0^1 |K_n(r) f(r) r|^2 r^{-2\beta} r \;dr
\]
 Since $0<r<1$, we may approximate $K_n$ with its asymptotic expansion near the origin. For $n \geq 1$ this leads to
\begin{equation}\label{e:norm_h_n}
  \|h \|_{L^2(-\infty,0)} \leq C \int_0^1 |f(r)|^2 r^{-2\beta -2n +2} \; r dr = \| f\|_{L^2_{r,\sigma+2,\gamma}(B_1)}
  \end{equation}

By combining the expressions \eqref{e:norm_w} and \eqref{e:norm_h_n} with the inequality \eqref{e:ineq_w} we obtain
the desired result, 
$$E^2 \leq \|v\|_{L^2_{r, \sigma+n,\gamma}(B_1)} \leq \| f\|^2_{L^2_{r,\sigma+2,\gamma}(B_1)}.$$

Next, we show that $D^2 \leq  \| f\|^2_{L^2_{r,\sigma+2,\gamma}(B_1)}$ using a similar approach.
The asymptotic approximation for the function $K_n(r)$ near the origin gives
\[D^2 \leq c \int_0^1 \left[ r^{-n} \int_0^r I_n(\rho) f(\rho) \rho \;d \rho \right]^2 r^{2 \sigma} \; rdr\]
Letting $g(\rho) = I_n(\rho) f(\rho) \rho $ and $v = \int_0^r g(\rho) \; \rho$, the integral on the right is equivalent to
\[ \|r^{-n} v \|^2_{L^2_{r, \sigma,\gamma}(B_1)} = \| v \|^2_{L^2_{r, \sigma-n,\gamma}(B_1)}.\]
To show $\|  v \|^2_{L^2_{r, \sigma-n,\gamma}(B_1)} \leq \| f\|^2_{L^2_{\sigma+2,\gamma}(B_1)}$ we again use the change of coordinates \eqref{e:change_coor}, but now with $\beta = - [\sigma -n +1]$, leading to equation \eqref{e:w_eq}. As before, this ode has solutions of the form \eqref{e:2_sol} and we find that
\begin{equation}\label{e:boundw1}
\| w \|_{L^2(-\infty,0)} \leq c(\beta) \| h\|_{L^2(-\infty,0)},
  \end{equation}
  provided $\beta  = - [\sigma -n +1]>0$. Since by assumption $\sigma<-n-1< n-1$, this condition holds.
  
Notice that with this choice of $\beta$,
\[ \| w\|_{L^2(-\infty,0)}= \| v\|_{L^2_{r,\sigma-n, \gamma}(B_1)}.\]
At the same time, for $n \geq 1$ we may use the asymptotic form of $I_n(r)$ near the origin to obtain
\[ \| h\|_{L^2(-\infty,0)} \leq c \| f\|^2_{L^2_{\sigma+2,\gamma}(B_1)}\]
and the result for $D^2$ then follows from inequality \eqref{e:boundw1}.

{\bf Step 1.3:} We now treat the case $n =0$ and start  with the term $D$. In this case,
\[D^2 \leq c \int_0^1 \left[ \log r \int_0^r  f(\rho) \rho \;d \rho \right]^2 r^{2 \sigma} \; rdr\]
since $K_0(r) \sim \rmO(\log(r))$ and $I_0(r) \sim \rmO(1)$ for $r$ near the origin.
The above inequality can then be written as
\[D^2 \leq  \| \log r \; v \|_{L^2_{r,\sigma,\gamma}(B_1)}\]
where $v = \int_0^r g(\rho) \; d\rho$ and $g(\rho) = f(\rho) \rho$.
Using the alternative change of coordinates
\[ \tau = \log r, \quad \tau w(\tau) = v(\rme^\tau) \rme^{-\beta \tau}, \quad h(\tau) = g(\rme^{\tau}) \rme^{-\beta \tau} \rme^{\tau} \]
we find that $w$ satisfies the o.d.e.
$\partial_\tau ( \tau w(\tau) \rme^{\beta \tau}) = h(\tau) \rme^{\beta \tau}$.
This equation has solutions 
\[ \tau w(\tau) = \int_{-\infty}^ \tau h(s) \rme^{\beta(s-\tau)}\;ds\]
so that for $\beta  = - [\sigma +1]>0$, Young's inequality gives us
\[ \| \tau w(\tau) \|_{L^2(-\infty,0)} \leq \frac{1}{\beta} \| h(\tau) \|_{L^2(-\infty, 0)}\]
Notice also that for this choice of $\beta$, 
\[ \| \log r \; v(r) \|_{L^2_{r,\sigma,\gamma}(B_1)} = \| \tau w(\tau) \|_{L^2(-\infty,0)} \]
\[ \| h(\tau) \|_{L^2(-\infty, 0)} \leq  \| f\|_{L^2_{r,\sigma+2, \gamma}(B_1)}\]
Giving us the desired inequality,
\[ D^2 \leq  \|\log r \; v\|_{L^2_{r,\sigma, \gamma}(B_1)} = \| \tau \; w\|_{L^2(-\infty,0)}\leq  \frac{1}{\beta}  \| h\|_{L^2(-\infty,0)} \leq \frac{1}{\beta}  \| f\|^2_{L^2_{\sigma+2,\gamma}(B_1)} \]

To bound the term  $E$ when $n =0$, 
noticing first that since $K_0$ is a decreasing function,  for values of  $r< \rho$ we have
\[ K_0(\rho) = K_0(r) \frac{K_0(\rho)}{K_0(r)} < K_0(r).\]
Therefore,
\begin{align*}
E^2 \leq &  \int_0^1\left[ \int_r^\infty K_0(\rho) f(\rho) \; \rho d\rho \right]^2 r^{2 \sigma} \;r dr\\
\leq  &  \int_0^1\left[ K_0(r) \int_r^\infty  f(\rho) \; \rho d\rho \right]^2 r^{2 \sigma} \;r dr\\
\leq  &  \int_0^1\left[ \log r \int_r^\infty  f(\rho) \; \rho d\rho \right]^2 r^{2 \sigma} \;r dr\\
\leq & D^2,
\end{align*}
where the first line follows from approximating  $I_0(r) \sim \rmO(1)$ for $r$ near the origin, while the second to  last line follows from $K_0(r) \sim -\log r $ for $r\sim 0$.

{\bf Step 2:} To bound the derivative $\partial_r u$ in $L^2_{r,\sigma+1,\gamma}(\R^2)$ we first use expression 
\eqref{e:inverseu} to obtain
\begin{equation}\label{e:derivative_inverseu}
\partial_r u(r) = I'_n(r) \int_r^\infty K_n(\rho) f(\rho) \rho \;d\rho +
  K'_n(r) \int_0^r I_n(\rho) f(\rho) \rho \;d\rho .
  \end{equation}
Notice that for $n \geq 1$, the modified Bessel functions satisfy,
\[
I'_n(r) = I_{n-1} - \frac{n}{r} I_n(r), \qquad K'_n(r) = K_{n-1} - \frac{n}{r} K_n(r),\]
so that  
\[ I'_n(r) \sim \rmO(I_n(r)/r)\qquad K'_n(r) \sim \rmO(K_n(r)/r)\quad \mbox{ as}\quad  r  \to 0.\]
Therefore, the results from Step 1 show that 
\begin{equation}\label{e:ineq_ur}
 \| \partial_r u \|_{L^2_{r,\sigma+1,\gamma}(\R^2)} \leq \|f \|_{L^2_{r,\sigma+2,\gamma}(\R^2)}
 \end{equation}
On the other hand, for $n=0$, 
\[ I_0'(r) = I_1(r) \sim \rmO(r) \qquad K'_0(r) = -K_1(r) \sim \rmO(1/r) \]
and again we find that inequality \eqref{e:ineq_ur} is satisfied.

{\bf Step 3:} Finally, we show that the operator has a trivial nullspace. 
This follows from noticing that the only elements in the kernel of the operator are the modified Bessel functions $K_n(r)$ and $I_n(r)$. However, since $I_n(r)$ grows exponentially, it does not live in $H^2_{r,\sigma, \gamma}(\R^2)$ for any value of $\gamma \in \R$. On the other hand, a short calculations shows that $K_n(r)$ are part of this weighted Sobolev space if $\sigma > n -1$. Therefore these functions are not part of the domain of the operator when $\sigma<-n-1$.

\end{proof}

\begin{Lemma}\label{l:invertible_helmoltz}
Fix $n \in \N $ and take $-n-1 < \sigma < n-1$.
Then, the operator
\[ \Delta_n - \Id: H^2_{r,\sigma, \gamma}(\R^2) \longrightarrow L^2_{r,\sigma +2, \gamma}(\R^2)\]
is invertible.
\end{Lemma}

\begin{proof}
To prove the operator is surjective one can follow the same steps as in the proof of Lemma \ref{l:surjective_helmholtz}, except when dealing with the inequality $\|u\|_{L^2_{r,\sigma,\gamma}(B_1)}\leq \|f \|_{L^2_{r,\sigma+2,\gamma}(B_1)}$. In this case we again write
$\|u\|_{L^2_{r,\sigma,\gamma}(B_1)}\leq E+D$ with 
\begin{align*}
E^2 = & \int_0^1\left[I_n(r)  \int_r^\infty K_n(\rho) f(\rho) \; \rho d\rho \right]^2 r^{2 \sigma} \;r dr\\
D^2 = & \int_0^1 \left[K_n(r)  \int_0^r I_n(\rho) f(\rho) \; \rho d\rho \right]^2 r^{2 \sigma} \;r dr
\end{align*}
and show that $E^2 \leq \|f \|_{L^2_{r,\sigma+2, \gamma}(\R^2)}$ directly, without having to change the order of integration.

Letting $g(\rho) = - K_n(\rho) f(\rho) \rho $ and $v (r) = \int_\infty^r g(\rho) \; d\rho,$ we have that $v$ solves 
\begin{equation}\label{e:v_eq}
 \partial_r v = g(r) \qquad \lim_{r \to \infty} v(r) =0.
\end{equation}
Since $I_n(r) \sim \rmO(r^n)$ for $ r \sim 0$, our aim is to show that
\[ E^2 \leq \| r^n v(r) \|^2_{L^2_{r,\sigma,\gamma}(B_1)} =\|  v(r) \|^2_{L^2_{r,\sigma+n ,\gamma}(B_1)} \leq \|f \|_{L^2_{r,\sigma+2, \gamma}(R_2)}.  \]
As before, we use the change of coordinates \eqref{e:change_coor} with $\beta = -(\sigma+n+1)$ to arrive at
\[ \partial_\tau(w \rme^{\beta \tau} ) = h(\tau ) \rme^{\beta \tau}, \qquad \lim_{\tau \to \infty} w(\tau) \rme^{\beta \tau} =0.\]
This o.d.e. then has solutions 
\[ w(\tau) = \int_\infty^\tau h(s)\rme^{\beta(s-\tau)} \;ds\]
which are in $L^2(-\infty,0)$ provided $\beta <0$.  This holds for our choice $\sigma > -n-1$, so that by Young's inequality we have that
\[ \| w\|_{L^2(-\infty,0) } \leq - \frac{1}{\beta} \| h\|_{L^2(\R)}.\]
To relate $h$ back to $f$, when $n \geq 1$ we determine
\begin{align*}
 \|h \|_{L^2(\R)} & = \int_0^\infty |g|^2 r^{-2 \beta} r^2 \frac{1}{r} \;dr \\
 & = \int_0^1 |K_n(r) f(r) r|^2 r^{-2\beta} r \;dr  +  \int_1^\infty |K_n(r) f(r) r|^2 r^{-2\beta} r \;dr\\
 & \leq C\int_0^1 | f(r) r|^2 r^{-2\beta-2n} r \;dr + \int_1^\infty | f(r) r|^2 \langle r \rangle^{2\gamma-2} r \;dr\\
 & \leq C \|f\|_{L^2_{r,\sigma+2,\gamma}(\R^2)}
\end{align*}
where the second to last inequality follows from the algebraic behavior of $K_n(r)$ near the origin and its exponential decay at infinity. The last line is a consequence of $\beta  = -(\sigma +n +1)$.
The definition of $w$ and our choice of $\beta$ then imply
\[ E^2 \leq \|  v(r) \|^2_{L^2_{r,\sigma+n ,\gamma}(B_1)} = \|w\|_{L^2(-\infty,0)} \leq -\frac{1}{\beta} \|h\|_{L^2(\R)}  \leq  -\frac{1}{\beta} \|f \|_{L^2_{r,\sigma+2, \gamma}(R_2)}.  \]

To show the bound  $D \leq \|f \|_{L^2_{r,\sigma+2,\gamma}(\R^2)}$, one can  follow the same analysis as in Lemma \ref{l:surjective_helmholtz}.

To show the operator is injective, notice that the only possible elements in the kernel of the operator are the modified Bessel functions $K_n(r)$. However, these functions are not in $L^2_{r,\sigma,\gamma}(\R^2)$ for our choice $\sigma < n-1$.

\end{proof}

%%%%%%%%%%%%%%%%%%%%%%%%%%%%%%%%%%%%%%%%%%%%%%%
\begin{Proposition}\label{p:injective}
Let $\sigma <-1$, $\gamma \in \R$. Then, the operator
\[ \Delta - \Id: H^2_{\sigma, \gamma}(\R^2) \longrightarrow L^2_{\sigma +2, \gamma}(\R^2)\]
is injective, with closed range, and cokernel that depends on $\sigma$. More precisely, for $n \in \N \cup \{0\}$ and with $-n-2 < \sigma < -n-1$, we have that the cokernel is spanned by the set,
\[ M_n =\left\{ K_m(r) \cos(m \theta),  K_m(r) \sin(m \theta) \; \Big| \quad m = 0,1,2, \cdots n \right\}.\]

where $K_m(r)$ denotes the modified Bessel function of the second kind and of order $m$.
\end{Proposition}
\begin{proof}
To analyze the operator we use the decomposition of $H^2_{\sigma, \gamma}(\R^2)$ into the direct sum $\oplus h^2_{n,\sigma, \gamma}$ and write
\begin{align} \nonumber
(\Delta - \Id) u = & f\\
\label{e:decomp}
\sum_m ( \Delta_m - 1)u_m(r)\; \rme^{\rmi m \theta} =& \sum_m f_m(r) \;\rme^{\rmi m \theta} 
\end{align}
Fix $n \in \N \cup \{0\}$ and take $\sigma \in (-n-2, -n-1)$. Then, for all those $m \in \Z$ satisfying 
$$-|m|-2 \leq -n -2<   \sigma,  $$
 Lemma \ref{l:invertible_helmoltz} shows that the operators $( \Delta_m - 1): H^2_{r,\sigma,\gamma}(\R^2) \longrightarrow L^2_{r,\sigma+2,\gamma}(\R^2)$ are invertible. On the other hand, for those $m \in \Z$ satisfying 
 $$\sigma< -n-1 \leq  -|m|-1,$$
  Lemma \ref{l:surjective_helmholtz} shows that these same operators are injective, have a closed range, and have as cokernel the modified Bessel function $K_m(r)$.
  
  Let \[ R = \{ f \in L^2_{\sigma+2,\gamma}(\R^2) \mid \langle f, h \rangle =0, \quad \forall h \in M_n\}.\]
where $M_n$ is defined as in the statement of this proposition. Then, because $\sigma < -1$, each function $h \in M_n$
defines a bounded linear functional in $L^2_{\sigma +2, \gamma}(\R^2)$. 
The space $R$ is therefore the union of the nullspaces of these functionals and is therefore closed. The decomposition \eqref{e:decomp} together with the results of Lemma \ref{l:surjective_helmholtz}
 then show that this subspace is the range of  our operator 
 $ \Delta - \Id: H^2_{\sigma, \gamma}(\R^2) \longrightarrow L^2_{\sigma +2, \gamma}(\R^2)$.

\end{proof}

%%%%%%%%%%%%%%%%%%%%%%%%%%%%%
We now use the results of  Proposition \ref{p:injective} and duality to prove Fredholm properties of the operator when  $\sigma >-1$.

\begin{Proposition}\label{p:surjective}
Let $\sigma >-1$, $\gamma \in \R$. Then, the operator
\[ \Delta - \Id: H^2_{\sigma, \gamma}(\R^2) \longrightarrow L^2_{\sigma +2, \gamma}(\R^2)\]
is surjective with kernel that depends on $\sigma$. More precisely, for $n \in \N \cup \{0\}$ and with $n-1 < \sigma < n$, we have that the kernel is spanned by the set,
\[ M_n =\left\{ K_m(r) \cos(m \theta),  K_m(r) \sin(m \theta) : \quad m = 0,1,2, \cdots n \right\}.\]

where $K_m(r)$ denotes the modified Bessel function of the second kind and of order $m$.
\end{Proposition}
\begin{proof}
Fix $\gamma \in $, $n \in \N \cup\{0\}$ and take $-n-1< \sigma <n$.
Consider the extended operator,
\begin{equation}\label{e:larger}
 (\Delta - \Id)^\dagger :L^2_{\sigma, \gamma}(\R^2) \longrightarrow H^{-2}_{\sigma +2, \gamma}(\R^2).
 \end{equation}
Then, by Proposition \ref{p:injective} its dual,
\[ (\Delta - \Id)^* :H^2_{-(\sigma+2), -\gamma}(\R^2) \longrightarrow L^2_{-\sigma, -\gamma}(\R^2)\]
is injective, with closed range, and cokernel given by $M_n$. 
By duality, the operator \eqref{e:larger} is surjective with kernel spanned by $M_n$.
To show that the `smaller' operator
 $(\Delta - \Id) : H^2_{\sigma, \gamma}(\R^2) \longrightarrow L^2_{\sigma +2, \gamma}(\R^2)$
 has closed range, consider  a sequence $\{ u_n\}$ in $H^2_{\sigma, \gamma}(\R^2)$ satisfying
 \[ \| u_n\|_{H^2_{\sigma,\gamma}(\R^2) } =1 \qquad \| (\Delta - \Id) u_n \|_{L^2_{\sigma+2,\gamma}(\R^2)} \to 0.\]
 Because of the embedding $L^2_{\sigma+2, \gamma}(\R^2) \subset H^{-2}_{\sigma +2, \gamma}(R^2)$ we also see that
  \[ \| (\Delta - \Id) u_n \|_{H^{-2}_{\sigma+2,\gamma}(\R^2)} \to 0.\]
Then, since \eqref{e:larger} has closed range, we can find $v_n \in M_n$  such that
 \[ \| u_n - v_n \|_{L^2_{\sigma, \gamma}(\R^2) } \to 0.\]
These two results, together with inequality 
 \[ \| u\|_{H^2_{\sigma, \gamma}(\R^2)} \leq c \left(  \| (\Delta - \Id) u  \|_{L^2_{\sigma+2,\gamma}(\R^2)} + \|u\|_{L^2_{\sigma, \gamma}(\R^2)}\right)
 \]
 shown in Lemma \ref{l:interpolation_2} in the Appendix, imply that $\| u_n -v_n \|_{H^2_{\sigma, \gamma}(\R^2)} \to 0$. As a result, we conclude that the operator has closed range and kernel spanned by $ M_n$.
\end{proof}

%%%%%%%%%%%%%%%%%%%%%%%%%%%%%%%%%%%%%%%%
%%%%%%% DELTA -1/R^2
%%%%%%%%%%%%%%%%%%%%%%%%%%%%%%%%%%%%%%%%

\section{Fredholm Properties for $\Delta -1/r^2$}\label{s:laplace}
In this section we prove Fredholm properties for the operator
 \[\Delta -1/r^2: M^{2,2}_{\sigma,\gamma}(\R^2) \longrightarrow L^2_{\sigma+2, \gamma+2}(\R^2),\]
summarized in Theorem \ref{t:Laplace} stated in the introduction.
We split the proof into four cases depending on the value of the weights $\sigma $ and $\gamma$. These
four results are summarized in Propositions \ref{p:fredholm1}- \ref{p:fredholm4}.

%%%%%%%%%%%%%SURJECTIVE %%%%%%%%%%%%%%%%%

\begin{Proposition}\label{p:fredholm1}
Let $n \in \Z$, $\sigma >-2$ and $\gamma<0$. Define $q(n)=\sqrt{n^2+1}$ and  assume  that $\sigma +1 \neq \pm q(n)$
and $\gamma +1 \neq \pm q(n) $. Then, the operator
$$ \Delta -1/r^2: M^{2,2}_{\sigma,\gamma}(\R^2) \longrightarrow L^2_{\sigma+2, \gamma+2}(\R^2).$$
is surjective with a finite dimensional kernel that depends on $\gamma$ and $\sigma$. More precisely,
letting $k,n, m \in \N \cup \{0\}$,
we have that
\begin{itemize}
\item if $ - q(n+1) < \gamma +1< - q(n)$ the kernel includes
\[ N_n = \left\{ r^{q(k)}\cos ( k \theta) ,\; r^{q(k)} \sin( k \theta) \Big| \quad k = 0,1,2, \cdots n \right \}.\]
 \item In addition, if $  q(m) < \sigma +1< q(m+1)$, the kernel is also composed of
 \[ M_m =\left\{ r^{-q(k)}\cos ( k \theta) ,\; r^{-q(k)} \sin( k \theta) \Big|  \quad k = 0,1,2, \cdots m \right \}.\]

\end{itemize}
On the other hand, if $\gamma +1 = \pm q(n) $ or $\sigma +1 \neq \pm q(n)$ for some $n \in \Z$, the operator does not have a closed range.
\end{Proposition}

\begin{proof}
Let $f \in L^2_{\sigma+2, \gamma+2}(\R^2)$ and suppose $(\Delta - \frac{1}{r^2})u =f$. We may then write,
\[ \sum_n \Delta_{n^2+1} u_n \rme^{\rmi n \theta} = \sum_n f_n \rme^{\rmi n \theta}.\]
where $ \Delta_{n^2+1} = \partial_{rr}  + \frac{1}{r} \partial_r  - \frac{n^2 +1}{r^2} $.

In what follows we show that for our choice of $\sigma$ and $\gamma$, the following inequalities hold
\begin{align}\label{e:ineq_1}
\| u_n \|_{M^{1,2}_{r, \sigma, \gamma}(\R^2 \setminus B_1) }& \leq C \| f_n \|_{L^2_{\sigma+2, \gamma+2}(\R^2 \setminus B_1)  }\\ 
\label{e:ineq_2}
 \| u_n \|_{M^{1,2}_{r, \sigma, \gamma}(B_1) } &\leq C \| f_n \|_{L^2_{\sigma+2, \gamma+2}(B_1)},
\end{align}
where the constant $C$ is independent of $n$. These inequalities together with the equivalence \eqref{e:directsum1} stated in Section \ref{s:preliminaries} then imply that 
\[ \| u \|_{M^{1,2}_{ \sigma, \gamma}(\R^2)} \leq C \| f \|_{L^2_{\sigma+2, \gamma+2}(\R^2)  }.\]

Thus, to prove the surjectivity of the operator we are left with showing that
$$\| D^\alpha u\|_{L^2_{\sigma+2, \gamma+2}(\R^2)} \leq C \| f \|_{L^2_{\sigma+2, \gamma+2}(\R^2)}$$
 for all multi-indices $\alpha$, satisfying
$|\alpha|=2$. However, notice that this last point easily follows from the equation, which gives
\[ \| \Delta u \|_{L^2_{\sigma +2, \gamma+2}(\R^2) } \leq \| f\|_{L^2_{\sigma+2, \gamma+2}(\R^2)} + \| u\|_{L^2_{\sigma, \gamma}(\R^2)}\]
and the interpolation inequality,
\[\| D^2 u\|_{L^2_{\sigma+ 2,\gamma +2}(\R^2)} \leq C \left[ \| \Delta u \|_{L^2_{\sigma +2, \gamma+2}(\R^2)} + 
 \| D u \|_{L^2_{\sigma +1, \gamma+1}(\R^2)} \right] \]
proved as inequality \eqref{e:interpolation} in lemma \ref{l:interpolation} in the Appendix.

We now prove inequality \eqref{e:ineq_1}. Choosing the rescaling $\tau = \ln r$ with $r \in [1,\infty)$, and the change of variables
\[ u_n = w \rme^{-\beta \tau}, \quad \beta = \gamma+1, \quad \mbox{and} \quad f_n = g\rme^{-\alpha \tau}, \quad \alpha = \gamma +3,\]
the equation $\Delta_{n^2+1}u_n = f_n$ can be written as
\begin{equation}\label{e:ode_w}
 w_{\tau \tau} -2\beta w_\tau +\beta^ 2 w -(n^2+1) w = g,
 \end{equation}
and we find that inequality \eqref{e:ineq_1} is equivalent to
\begin{equation}\label{e:ineq_w_gamma}
 \| w\|_{H^1[0,\infty)} \leq C \|g \|_{L^2[0,\infty)}.
 \end{equation}

The second order o.d.e. \eqref{e:ode_w} can be written as a first order system with constant coefficients,
\begin{equation}\label{e:first_order} \frac{dW}{d\tau} = A W + G(\tau),
\end{equation}
where $W= (w , w_\tau)$, $G = (0,g)$ and $ A \in \R^{2\times 2}$.
The eigenvalues of matrix $A$ are 
\[ \lambda_{\pm} = \beta \pm \sqrt{n^2 +1},\]
and the equation has a solutions $w$ in $H^1[0,\infty)$ provided at least one of the two eigenvalues $\lambda_{\pm}$ is negative.

 This condition is satisfied for all $n \in \Z$ if $\beta < 1$, or equivalently, if $\gamma <0$.
In this case, the variation of constants formula gives us one possible solution,
\[  W( \tau ) =    \rme^{ \lambda_- \tau} W( 0 )  + \int_{0}^\tau \rme^{\lambda_- (\tau - t)} P^s G(t) \;dt\]
where  $P^s$ is the projection onto the eigenspace of the negative eigenvalue $\lambda_- = \beta - \sqrt{n^2+1}$.
With out loss of generality we can pick $W(0) =0$ since, as shown below, the operator has a kernel, which we can use to change this value.
Then, inequality \eqref{e:ineq_w_gamma} 
can be obtained using Young's inequality for convolutions, from which we deduce that the constant
  $C =  \frac{-1}{\beta - \sqrt{n^2+1}} < \frac{-1}{\beta-1} = \frac{-1}{\gamma}$.

The same argument can be used to obtain inequality \eqref{e:ineq_2}. In this case we choose 
$\tau = \ln r$ with $r \in (0,1]$, and the change of variables
\[ u_n = w \rme^{-\beta \tau}, \quad \beta = \sigma+1, \quad \mbox{and} \quad f_n = g\rme^{-\alpha \tau}, \quad \alpha = \sigma +3,\]
We arrive at the same differential equation \eqref{e:ode_w}, but now inequality \eqref{e:ineq_2}  is given by
\begin{equation}\label{e:ineq_w_sigma}
 \| w\|_{H^1(-\infty,0]} \leq C \|g \|_{L^2(-\infty,0]}.
 \end{equation}

In order to obtain a solution $w$ in $H^1(-\infty,0]$ at least one of the eigenvalues $\lambda_{\pm}$ has to be positive. This condition holds for all $n \in \Z$ provided $\beta >-1$, or equivalently if $\sigma >-2$.
The solution is then given by the variation of constants
\[ W(\tau ) = \rme^{\lambda_+ \tau} W(0) +  \int_{0}^\tau \rme^{\lambda_+(\tau -t)} P^u G(t)  \;dt,\]
where we use the same notation as before, but now $P^u$ denotes the projection onto the eigenspace of the positive eigenvalue $\lambda_+ = \beta + \sqrt{n^2+1}$.
As before, we may pick $W(0)=0$ and, using Young's inequality for convolutions, we arrive at inequality \eqref{e:ineq_w_sigma} with constant $C= \frac{1}{\beta + \sqrt{n^2+1}} < \frac{1}{\beta + 1} = \frac{1}{\sigma +2}$.

Notice that when $\sigma = \pm \sqrt{n^2+1}$ or $\gamma = \pm \sqrt{n^2+1}$, one of the eigenvalues $\lambda_{\pm}$ is zero. We can then use elements in this center eigenspace to construct Weyl sequences, proving that the operator does not have a closed range.

To finish the proof of the proposition we just need to determine the elements of the kernel. It is straightforward to check that the functions $ r^{\pm q(n)}$ with $q(n) = \sqrt{n^2+1}$, $n \in\Z $, solve the equation
\[ \Delta_{n^2+1} r^{\pm q(n)} =0.\]
As a result,
the real and imaginary parts of $r^{\pm q(n)} \rme^{\rmi n \theta}$ 
are elements in the kernel of $\Delta - 1/r^2$. However, 
$ \mathrm{Re}[ r^{ q(n)} \rme^{\rmi n \theta}], \quad \mathrm{Im} [ r^{ q(n)} \rme^{\rmi n \theta}]$ 
 are only in the space $M^{2,2}_{\sigma, \gamma}$ provided
\[ \gamma +1< - \sqrt{n^2+1},\]
while 
$ \mathrm{Re}[ r^{ -d(n)} \rme^{\rmi n \theta}], \quad \mathrm{Im}[ r^{- d(n)} \rme^{\rmi n \theta}]$ are in $M^{2,2}_{\sigma, \gamma}$ if 
\[ \sigma +1>   \sqrt{n^2+1}.\]

\end{proof}

%%%%%%%%%%%%% INJECTIVE %%%%%%%%%%%%%

\begin{Proposition}\label{p:fredholm2}
Let $n \in \Z$, $\sigma < -2$ and $\gamma> 0$.  Define $q(n) = \sqrt{n^2+1}$ and assume that $\sigma +1 \neq \pm q(n)$ and 
$\gamma +1 \neq \pm q(n)$. Then, the operator
$$ \Delta -1/r^2: M^{2,2}_{\sigma,\gamma}(\R^2) \longrightarrow L^2_{\sigma+2, \gamma+2}(\R^2).$$
is injective with a finite dimensional co-kernel that depends on $\gamma$ and $\sigma$. More precisely,
letting $k, n, m \in \N \cup \{0\}$, we have that
\begin{enumerate}[i)]
\item if $  q(n)< \gamma +1 < q(n+1)$ the co-kernel includes
\[ N_n = \left\{ r^{q(k)} \cos( k \theta),\; r^{q(k)} \sin( k \theta) \Big | \;  \quad k = 0,1,2, \cdots n \right \}.\]
 \item In addition, if $ - q(m+1)  < \sigma +1< - q(m)$, the co-kernel is also composed of
 \[ M_m =\left\{  r^{-q(k)} \cos( k \theta),\; r^{-q(k)} \sin( k \theta) \Big | \; \quad k = 0,1,2, \cdots m \right\}.\]
\end{enumerate}
If $\gamma +1 = \pm q(n)$ or $\sigma +1 \neq \pm q(n) $ for some $n \in \Z$, the operator does not have a closed range.
\end{Proposition}

\begin{proof}
Fix $n, m \in \N \cup \{0\}$ and assume that the weights $\gamma$ and $\sigma$ satisfy  inequalities i) and ii), as stated in the Proposition. Define
\[ R = \{  f\in L^2_{\sigma+2, \gamma+2}(\R^2) \mid \langle f, h\rangle =0, \; h  \in N_n \cup M_n  \}\]
and notice that this set  is closed. Indeed, this follows since the functions $h \in N_n \cup M_m$ define bounded linear functionals, so that the set $R$ corresponds to the union of their nullspace.

We now show that the range of the operator is given by $R$, proving the statement of the Proposition.
Let $f \in R$ and write 
\[ (\Delta - \frac{1}{r^2} ) u =  \sum_k \Delta_{k^2+1} u_k \; \rme^{\rmi k \theta} = \sum_k f_k \rme^{\rmi k \theta}. \]
 As in the proof of Proposition \ref{p:fredholm1}, we show that  for our choice of $\sigma$ and $\gamma$, the inequalities 
\begin{align}\label{e:ineq_11}
\| u_k \|_{M^{1,2}_{r, \sigma, \gamma}(\R^2 \setminus B_1) }& \leq C \| f_k \|_{L^2_{\sigma+2, \gamma+2}(\R^2 \setminus B_1)  }\\ 
\label{e:ineq_12}
 \| u_k \|_{M^{1,2}_{r, \sigma, \gamma}(B_1) } &\leq C \| f_k \|_{L^2_{\sigma+2, \gamma+2}(B_1)},
\end{align}
hold, where now the constant $C$  depends on the choice of fixed $n$ and $m$.  The result of the Proposition then follow by invoking relation \eqref{e:directsum1}, inequality \eqref{e:interpolation} and the equation $(\Delta - \frac{1}{r^2} ) u = f$.

To show inequality \eqref{e:ineq_11} we again use the change of coordinates $\tau = \ln r$, with $r \in [1,\infty)$, and take
\[ u_k = w \rme^{-\beta \tau}, \quad \beta = \gamma+1, \quad \mbox{and} \quad f_k = g\rme^{-\alpha \tau}, \quad \alpha = \gamma +3.\]
The equation $ \Delta_{k^2+1} u_k =f_k$ can then be written as a first order system with constant coefficients (see \eqref{e:first_order}), with corresponding eigenvalues, 
$ \lambda_{\pm} = \beta \pm \sqrt{k^2+1}.$

For values of $\gamma$ satisfying
$$q(k) < \gamma < q(k+1),$$
 where $q(k) = \sqrt{k^2+1}$, we have that the eigenvalue $\lambda_- = \beta - q(k) <0$ for all $
|k| \geq |n|+1$. Therefore, for these $k$ values, we may write a solution to system  \eqref{e:first_order} using the variation of constants formula
\[  \rme^{ -\lambda_- \tau}W(\tau) - \lim_{s \to -\infty}  \rme^{- \lambda_- s} W( s) =  \int_{-\infty }^\tau \rme^{- \lambda_-  t} P^s G(t) \;dt,\]
where $P^s$ is the projection onto the eigenspace of $\lambda_-$.
Because by assumption $\sigma <-2$, the functions $w$ and $w_\tau$ are $ \in L^1(-\infty,0)$, so that the limit appearing in the above expression tends to zero. We may therefore write
\begin{equation}\label{e:vcf1}
  W( \tau ) =  \int_{-\infty}^\tau \rme^{\lambda_- (\tau - t)} P^s G(t) \;dt.
  \end{equation}
To show that  $w \in H^1(0, \infty)$ we use Young's inequality, resulting in expression \eqref{e:ineq_11} with $C = \frac{-1}{\gamma +1 - q(k) } \leq \frac{-1}{\gamma +1 - q(n+1)}$.
 
On the other hand, notice that for integer values satisfying $ |k| < |n|+1$, both eigenvalues $\lambda_{\pm}$ are positive. Consequently,  the variation of constants formula shown above does not lead directly  to a solution $w$ in $H^1(0,\infty)$.
 To remedy this, we use the fact that $f \in R$ in order to construct an alternative solution.
 More precisely, since $G(s) = (0,g(s))$ and
\begin{align*}
 \langle f, r^{q(k)} \rme^{\rmi k \theta} \rangle = &\int_0^1 f_k  r^{q(k)}  r\;dr\\
0= & \int_{-\infty}^\infty g(\tau) \rme^{-\alpha \tau} \rme^{q(k) \tau} \rme^{2 \tau} \; d\tau\\
0 = & \int_{-\infty}^\infty g(\tau)  \rme^{-( \beta - q(k)) \tau}  \; d\tau\\
0 = & \int_{-\infty}^\infty g(\tau)  \rme^{-\lambda_- \tau}  \; d\tau
\end{align*}
we may re-write expression \eqref{e:vcf1} as
\[ W( \tau ) =  \int_{\infty}^\tau \rme^{\lambda_- (\tau - t)} P^s G(t) \;dt.\]
We can now apply Young's inequality  in order to arrive at \eqref{e:ineq_11}, this time  with  $C = \frac{1}{\gamma +1 - q(k)} \leq \frac{1}{\gamma+1 - q(n+1)}$. Putting both results together, we conclude that inequality \eqref{e:ineq_11} holds for all integers $k$, so long as we pick $C = \max\{  \frac{1}{\gamma+1 - q(n+1)}, \frac{-1}{\gamma +1 - q(n+1)}\}$.

Next, to show inequality \eqref{e:ineq_12} we let  $\tau = \ln r$, with $r \in (0,1]$, and define
\[ u_n = w \rme^{-\beta \tau}, \quad \beta = \sigma+1, \quad \mbox{and} \quad f_n = g\rme^{-\alpha \tau}, \quad \alpha = \sigma +3.\]
A similar analysis as above then leads to the desired result with the constant in the inequality satisfying $C = \max\{ \frac{-1}{\sigma+1 +q(m)}, \frac{1}{\sigma+1 + q(m+1)}\}$.

Finally, as in the previous proposition, notice that when $\sigma = \pm \sqrt{n^2+1}$ or $\gamma = \pm \sqrt{n^2+1}$, one of the eigenvalues $\lambda_{\pm}$ is zero. We can then use elements in this center eigenspace to construct Weyl sequences, proving that the operator does not have a closed range.

\end{proof}

%%%%%%%%%%%%% MIXED  %%%%%%%%%%%%%
Similar proofs to the ones shown above lead to the next two propositions.

\begin{Proposition}\label{p:fredholm3}
Let $n \in \Z$, $\sigma > -2$ and $\gamma> 0$. Let $q(n) = \sqrt{n^2+1}$ and assume that $\sigma +1 \neq \pm q(n)$ and 
$\gamma +1 \neq \pm q(n) $. Then, the operator
$$ \Delta -1/r^2: M^{2,2}_{\sigma,\gamma}(\R^2) \longrightarrow L^2_{\sigma+2, \gamma+2}(\R^2).$$
is Fredholm with a finite dimensional kernel and co-kernel that depends on $\gamma$ and $\sigma$.
More precisely,
letting $k, n, m \in \N \cup \{0\}$ we have that
\begin{enumerate}[i)]
\item if $  q(n) < \gamma +1<q(n+1)$ the co-kernel includes
\[ N_n = \left\{ r^{q(k)} \cos( k \theta),\; r^{q(k)} \sin( k \theta) \Big | \;  \quad k = 0,1,2, \cdots n \right \}.\]
 \item In addition, if $  q(m)< \sigma +1<q(m+1)$, the kernel is composed of
 
 \[ M_m =\left\{  r^{-q(k)} \cos( k \theta),\; r^{-q(k)} \sin( k \theta) \Big | \;  \quad k = 0,1,2, \cdots m \right\}.\]
\end{enumerate}
If $\gamma +1 = \pm q(n)$ or $\sigma +1 \neq \pm q(n)$ for some $n \in \Z$, the operator does not have a closed range.
\end{Proposition}

\begin{Proposition}\label{p:fredholm4}
Let $n \in \Z$, $\sigma < -2$ and $\gamma< 0$.  Define $q(n) = \sqrt{n^2+1}$ and assume that $\sigma +1 \neq \pm q(n) $ and 
$\gamma +1 \neq \pm q(n) $. Then, the operator
$$ \Delta -1/r^2: M^{2,2}_{\sigma,\gamma}(\R^2) \longrightarrow L^2_{\sigma+2, \gamma+2}(\R^2).$$
is Fredholm with a finite dimensional kernel and co-kernel that depends on $\gamma$ and $\sigma$. 
More precisely,
letting $k, n, m \in \N \cup \{0\}$ we have that
\begin{enumerate}[i)]
\item if $  -q(n+1) < \gamma +1 <- q(n)$ the kernel includes
\[ N_n = \left\{ r^{q(k)} \cos( k \theta),\; r^{q(k)} \sin( k \theta) \Big | \;  \quad k = 0,1,2, \cdots n \right \}.\]
 \item In addition, if $ - q(m+1) < \sigma +1< - q(m) $, the co-kernel is also composed of
 \[ M_m =\left\{  r^{-q(k)} \cos( k \theta),\; r^{-q(k)} \sin( k \theta) \Big | \;  \quad k = 0,1,2, \cdots m \right\}.\]
\end{enumerate}
If $\gamma +1 = \pm q(n) $ or $\sigma +1 \neq \pm q(n) $ for some $n \in \Z$, the operator does not have a closed range.
\end{Proposition}

%%%%%%%%%%%%%%%%%%%%%%%%%%%%%%%%%%%%%
%%%%%%%%%%%%%APPENDIX %%%%%%%%%%%%%%%%%%
%%%%%%%%%%%%%%%%%%%%%%%%%%%%%%%%%%%%%

\section{Appendix}

\begin{Lemma}\label{l:interpolation}
Let  $\sigma , \gamma \in \R$, and take $u \in M^{1,2}_{\sigma, \gamma}(\R^2) $ and $\Delta u \in L^2_{\sigma+2,\gamma+2}(\R^2)$. Then, there exists a constant $C = C(\sigma, \gamma) >0$, 
such that,
\begin{equation}\label{e:interpolation}
 \|D^2 u \|_{L^2_{\sigma +2, \gamma+2}(\R^2)} \leq C \left[    \|\Delta u \|_{L^2_{\sigma +2, \gamma+2}(\R^2)} 
+ \|D u \|_{L^2_{\sigma +1, \gamma+1}(\R^2)}     \right],
\end{equation}
where $Du = \sum_{i =1}^2 u_{x_i}$ and $D^2u = \sum_{i,j =1}^2 u_{x_ix_j}$.
\end{Lemma}

\begin{proof}
Since $C_o^\infty(\R^2)$ is a dense subset of $M^{2,2}_{\sigma, \gamma}(\R^2)$, there is a smooth function  $\tilde{u}$  that is close to $u$ in the  $ M^{1,2}_{\sigma, \gamma}(\R^2) $  norm. For ease of notation, we let $u =\tilde{u}$ denote this smooth function.

Letting $\chi \in C^\infty(\R_+)$ denote a cut-off function with $\chi(|x|) =1$ for $|x|<1$ and $\chi(|x|) =0 $ for $|x|>2$, we write
\begin{align*}
 \| \Delta u\|^2_{L^2_{\sigma +2, \gamma +2} }  
 &= \int_{\R^2} | \Delta u|^2 \langle  x \rangle^{2 (\gamma +2)} b(|x|)^{2( \sigma +2)} \;dx,\\
 & \leq \int_{\R^2} \chi | \Delta u|^2 \; |x|^{2( \sigma +2)} \;dx +   \int_{\R^2}(1- \chi) | \Delta u|^2 \; |x|^{2 (\gamma +2)}  \;dx.
 \end{align*}

We start with the first integral, which leads to
\begin{align*}
\int_{\R^2} \chi | \Delta u|^2 \; |x|^{2( \sigma +2)} \;dx 
= & \sum_{i,j =1}^2 \int_{\R^2}  u_{x_i x_i} u_{x_j x_j} |x|^{2(\sigma +2)} \chi \; dx\\
= & \sum_{i,j =1}^2 \int_{\R^2}  u_{x_i x_j}  u_{x_j x_i }   |x|^{2(\sigma +2)} \chi  \; dx \\
 & + 2 (\sigma +2) \sum_{i,j =1}^2 \int_{\R^2}  u_{x_i x_j}  u_{ x_i }  \; x_j |x|^{2(\sigma +2)} \chi  \; dx\\
& - 2 (\sigma +2)  \sum_{i,j =1}^2 \int_{\R^2} u_{x_i }  u_{ x_j x_j}\; x_i |x|^{2(\sigma +1)} \chi  \; dx +loc.\\
\end{align*}

 where the first inequality follows from two applications of integration by parts, and $loc$ stands for smooth localized terms.
Using the notation stated in the Lemma  and re-arranging terms,
 \begin{align*}
\int_{\R^2}  |D^2u|^2  |x|^{2(\sigma +2)} \chi  \; dx
=& \int_{\R^2} \chi | \Delta u|^2 \; |x|^{2( \sigma +2)} \;dx \\
 & +  2 (\sigma +2) \sum_{i,j =1}^2 \int_{\R^2}  u_{x_i x_j}  u_{ x_i }  \; x_j |x|^{2(\sigma +1)} \chi  \; dx\\
& - 2 (\sigma +2)  \sum_{i,j =1}^2 \int_{\R^2} u_{x_i }  u_{ x_j x_j}\; x_i |x|^{2(\sigma +1)} \chi  \; dx +loc.\\
\end{align*}
Using Cauchy's inequality with $\eps$ we obtain,
\begin{align*}
\int_{\R^2}  |D^2u|^2  |x|^{2(\sigma +2)} \chi  \; dx
\leq & \int_{\R^2} \chi | \Delta u|^2 \; |x|^{2( \sigma +2)} \;dx \\
 +  4 |(\sigma +2)| & \left[  \eps  \int_{\R^2} \chi |D^2 u|^2  |x|^{2( \sigma +2)} \;dx \right.\\
 & \left. + \frac{1}{\eps}  \int_{\R^2} \chi |D u|^2  |x|^{2( \sigma +1)} \;dx, 
 \right] +loc.
  \end{align*}
so that by picking $\eps$ small enough we arrive at
\[ \| D^2 u\|_{L^2_{\sigma+2, \gamma+2}(B_1)} \leq C(\sigma) \left[  \| \Delta u\|_{L^2_{\sigma+2, \gamma+2}(B_1)} 
+ \| Du\|_{L^2_{\sigma+1, \gamma+1} (B_1)}   \right]. \]

A similar analysis leads to
\[ \| D^2 u\|_{L^2_{\sigma+2, \gamma+2}(\R^2 \setminus B_1)} \leq C(\gamma) \left[  \| \Delta u\|_{L^2_{\sigma+2, \gamma+2}(\R^2 \setminus B_1)} 
+ \| Du\|_{L^2_{\sigma+1, \gamma+1} (\R^2 \setminus B_1)}   \right], \]
from which the results of the Lemma then follow.
\end{proof}

%%%%%%%%%%%%%%%%%%%%%%%%%%%%%%%%%%%%%%%%%%%%

\begin{Lemma}\label{l:interpolation_2}
Let  $\sigma , \gamma \in \R$, and take $u \in H^{1,2}_{\sigma, \gamma}(\R^2) $ and $(\Delta  -\Id) u \in L^2_{\sigma+2,\gamma}(\R^2)$. Then, there exists a constant $C = C(\sigma, \gamma) >0$, 
such that,
\begin{equation}\label{e:interpolation_2}
 \| u \|_{H^2_{\sigma , \gamma}(\R^2)} \leq C \left[    \|(\Delta -\Id) u \|_{L^2_{\sigma +2, \gamma}(\R^2)} 
+ \|u \|_{L^2_{\sigma , \gamma}(\R^2)}     \right].
\end{equation}
\end{Lemma}
\begin{proof}
We show  that
\begin{equation}\label{e:ineq_delta}
  \| u \|_{H^2_{\sigma , \gamma}(\R^2)} \leq C\left[   \|\Delta  u \|_{L^2_{\sigma +2, \gamma}(\R^2)} 
+ \|u \|_{L^2_{\sigma , \gamma}(\R^2)}     \right].
\end{equation}
The result of the lemma then follow from the triangle inequality
\[  \|\Delta u \|_{L^2_{\sigma +2, \gamma}(\R^2)}  \leq  \|(\Delta -\Id) u \|_{L^2_{\sigma +2, \gamma}(\R^2)}  +  \| u \|_{L^2_{\sigma +2, \gamma}(\R^2)} \]
and the embedding $L^2_{\sigma,\gamma}(\R^2) \subset L^2_{\sigma+2,\gamma}(\R^2) $

To prove the result, we split inequality \eqref{e:ineq_delta} into two parts
\begin{align}\label{e:ineq_a}
 \| u \|_{H^2_{\sigma , \gamma}(B_1)} & \leq C\left[   \|\Delta  u \|_{L^2_{\sigma +2, \gamma}(B_1)} 
+ \|u \|_{L^2_{\sigma , \gamma}(B_1)}     \right] \\
\label{e:ineq_b}
 \| u \|_{H^2_{\sigma , \gamma}(\R^2\setminus B_1)} & \leq C\left[   \|\Delta  u \|_{L^2_{\sigma +2, \gamma}(\R^2\setminus B_1)} 
+ \|u \|_{L^2_{\sigma , \gamma}(\R^2\setminus B_1)}     \right].
\end{align}

We start with inequality \eqref{e:ineq_a}. Notice that from Lemma \ref{l:interpolation} we have that
\begin{equation}\label{e:ineqD2u}
 \|D^2 u \|_{L^2_{\sigma +2, \gamma}(B_1)} \leq C \left[    \|\Delta u \|_{L^2_{\sigma +2, \gamma}(B_1)} 
+ \|D u \|_{L^2_{\sigma +1, \gamma}(B_1)}     \right],
\end{equation} 
To bound the first derivatives, $Du$, we use the identity
\[ \int_{\R^2} \nabla u \cdot \nabla v  = \int_{\R^2} \Delta u v\]
which is valid for all $v \in C^2_0(B_R)$, with $B_R$ representing any fixed ball of radius $R$.
Letting $v = \chi u |x|^{2(\sigma+1)}$, where $\chi \in C^\infty(\R_+)$ is a cut-off function satisfying $\chi(|x|) = 1$ for $|x|\leq 1$ and $\chi(|x|) = 0$ for $|x|>2$, we can expand
\[ \nabla u \cdot \nabla v =  |x|^{2(\sigma+1)}  \nabla u \cdot \nabla( \chi u)+ 2(\sigma+1)  |x|^{2\sigma}  ( \nabla u \cdot x ) (\chi u).\]
Inserting this expression into the identity and rearranging terms we obtain,
\[ \int_{\R^2} \nabla u \cdot \nabla ( \chi u)\; |x|^{2(\sigma+1)} 
  = \int_{\R^2} \Delta u (\chi u)  |x|^{2(\sigma+1)} - 2(\sigma+1)   ( \nabla u \cdot x )  (\chi u)  |x|^{2\sigma}.\]
Then, using Cauchy's inequality, 
  \[ \| Du \|^2_{L^2_{\sigma+1, \gamma}(B_1)} \leq
C \left[  \| \Delta u \|_{L^2_{\sigma+2, \gamma}(B_1)} \| u \|_{L^2_{\sigma, \gamma}(B_1)}
+ 2(\sigma+1)  \| Du \|_{L^2_{\sigma+1, \gamma}(B_1)}  \| u \|_{L^2_{\sigma, \gamma}(B_1)}
\right]\]
followed by Cauchy's inequality with $\eps$,
\begin{align*}
 \| Du \|^2_{L^2_{\sigma+1, \gamma}(B_1)} \leq
C &\left[  \| \Delta u \|^2_{L^2_{\sigma+2, \gamma}(B_1)} + \| u \|^2_{L^2_{\sigma, \gamma}(B_1)} \right. \\
&\left.+ 2(\sigma+1) \eps  \| Du \|^2_{L^2_{\sigma+1, \gamma}(B_1)}  + 2(\sigma+1) \frac{1}{\eps} \| u \|^2_{L^2_{\sigma, \gamma}(B_1)}
\right],
\end{align*}
we arrive at
\begin{equation}\label{e:ineqDu} 
 \| Du \|^2_{L^2_{\sigma+1, \gamma}(B_1)} \leq C \left[   \| \Delta u \|^2_{L^2_{\sigma+2, \gamma}(B_1)} +
\| u \|^2_{L^2_{\sigma, \gamma}(B_1)}  \right], 
\end{equation}
by a suitable choice of $\eps$. Expressions\eqref{e:ineqD2u} and \eqref{e:ineqDu} then imply inequality \eqref{e:ineq_a}.

Similarly, to obtain relation \eqref{e:ineq_b}, we first use Lemma \ref{l:interpolation},
giving us
\begin{equation}\label{e:ineqD2u_b}
 \|D^2 u \|_{L^2_{\sigma +2, \gamma}(\R^2 \setminus B_1)} \leq C \left[    \|\Delta u \|_{L^2_{\sigma +2, \gamma}(\R^2 \setminus B_1)} 
+ \|D u \|_{L^2_{\sigma +1, \gamma-1}(\R^2 \setminus B_1)}     \right],
\end{equation} 
The inequality
\begin{equation}\label{e:ineqDu_b} 
 \| Du \|^2_{L^2_{\sigma+1, \gamma-1}(\R^2 \setminus B_1)} \leq C \left[   \| \Delta u \|^2_{L^2_{\sigma+2, \gamma}(\R^2 \setminus B_1)} +
\| u \|^2_{L^2_{\sigma, \gamma-2}(\R^2 \setminus B_1)}  \right], 
\end{equation}
then follows from a similar argument as the one used to derive \eqref{e:ineqDu}. 
The expression \eqref{e:ineq_b}  is then obtained by combining these last two inequalities, together with the embedding
\[ L^2_{\sigma,\gamma_1}(\R^2) \subset  L^2_{\sigma,\gamma_2}(\R^2) \]
which holds for all $\gamma_1> \gamma_2$.

\end{proof}

\bibliographystyle{plain}
\bibliography{dislocations}

\end{document}